\documentclass[11pt]{article}

\usepackage{amsmath,amscd,amsfonts}

\begin{document}

\newtheorem{theorem}{Theorem}[section]
\newtheorem{lemma}[theorem]{Lemma}
\newtheorem{proposition}[theorem]{Proposition}
\newtheorem{corollary}[theorem]{Corollary}
\newtheorem{definition}[theorem]{Definition}
\newtheorem{example}[theorem]{Example}

\newenvironment{proof}[1][Proof]{\begin{trivlist}
\item[\hskip \labelsep {\bfseries #1}]}{\end{trivlist}}
\newenvironment{remark}[1][Remark]{\begin{trivlist}
\item[\hskip \labelsep {\bfseries #1}]}{\end{trivlist}}

\newcommand{\qed}{\nobreak \ifvmode \relax \else
      \ifdim\lastskip<1.5em \hskip-\lastskip
      \hskip1.5em plus0em minus0.5em \fi \nobreak
      \vrule height0.75em width0.5em depth0.25em\fi}

\title{A Ganzstellensatz for semi-algebraic sets and a boundedness criterion for rational functions}

\author{Noa Lavi \\ \footnotesize{Department of Mathematics} \\ \footnotesize{The Hebrew University of Jerusalem} \\ \footnotesize{Givat Ram, Jerusalem 91904, Israel} \\                      \footnotesize{noa.lavi@mail.huji.ac.il} }
       
\maketitle

\begin{abstract}
Let $\langle K,\nu \rangle$ be a real closed valued field, and let $S\subseteq K^n$ be an open semi-algebraic set. Using tools from model theory, we find an algebraic characterization of rational functions which admit, on $S$, only values in the valuation ring. We use this result to deduce a criterion for a rational function to be bounded on an open semi algebraic subset of some irreducible variety over a real closed field or over an ordered field which is dense in its real closure.
\end{abstract}

\section {Introduction}

A ``Ganzstellensatz'' in algebraic geometry of valued fields is a theorem giving an algebraic characterization of rational functions whose values on a given definable set lie in the valuation ring (`ganze' elements). A Ganzstellensatz in $p$-adically closed fields was given as a $p$-adic analogue for Artin-Schreier theory of real closed fields, by Kochen and Cherlin, see \cite{ko} and \cite{cmta}. Later on, some Ganztsellens\"atze for rational functions on the valuation ring of some valued field was given by Prestel and Ripoll in \cite{PR}.\\ 
 A general model theoretic framework for proving Ganzstellens\"atze for theories of valued fields which are model complete was suggested by Haskell and Yaffe \cite{DY}, which could be seen as a model theoretic ``black box'' or ``oracle''. In the present paper we use this framework, in Proposition \ref{fw}, to obtain a Ganzstellensatz for open semi-algebraic sets over real closed valued fields. By further use of model theoretic properties, we also manage to deduce from it a criterion for boundedness of rational functions and polynomials over semi-algebraic sets over real closed fields, and even more general - over ordered fields which are dense in their real closure in respect to the order topology. \\

 The main result of the paper (Theorem~\ref{ganz}) states that if $K$ is a real closed valued field and $V$ is some irreducible real affine variety over it, then a rational function $h$ admits only integral values on an open semi-algebraic set $S_{\bar p}=\{ \bar x \in V\ : \ \bigwedge_{i=1}^m p_i(\bar x) > 0 \}$, exactly when $h$ is in the integral radical of the $O_K$-algebra generated by the set of functions of the form $\frac{1}{1+f}$, where $f$ is in the positive cone in $K(V)$ generated by the polynomials $p_1,...,p_m$ from $K[V]$ .  \\
In order to prove the above, one first has to show that all the elements in the candidate algebra indeed admit values only in the valuation ring, which is done in Lemma \ref{nearPos}. In order to show that it is sufficient, one needs to show that the ``sufficiency condition'' of the framework indeed holds. That could be summarized as giving an ordered valued field structure to the field of rational functions $K(V)$. That structure should be an extension of the structure on the ground field, in such a way that the formula defining the semi-algebraic set holds - $p_1,..,p_m$ shall be positive elements. This is done by some generalization of the Baer-Krull Theorem \cite{efr}. \\

We would like to explain further the relationship between the present paper and \cite{DY}.  The proof of theorem \ref{ganz} and \ref{ganz2} was done in the author's master thesis \cite{MA}. The starting point of this work was Theorem $4.12$ in \cite{DY} for functions in one variable. As a first stage for the master thesis, a generalization of that theorem to the case of one polynomial in many variables was proved using the ingredient from Baer-Krull \cite{efr}. A proof for that case was given as well in the published version of \cite{DY}. The next stage - and the main result of this paper - is a generalization to several polynomials $p_1,...,p_m$.\\
 
A natural question which arises in real algebraic geometry is about the boundedness of functions over semi-algebraic sets. That is, let $R$ be an arbitrary real closed field, and let $S$ be a semi-algebraic subset of $R^n$. How could we characterize the polynomials or the rational functions which are bounded on $S$? \\
There has been several works studying the ring of bounded polynomials on semi-algebraic sets. For example, by M. Schweighofer in \cite{markus} who studied the relationship of boundedness with sums of squares and positivity and proved a result about iterating the construction of the bounded polynomial rings on a semi-algebraic set. Another example is the work due to D. Plaumann and C. Scheiderer in \cite{dan}, who studied rather geometrical aspects. Here, we use the main result of the paper to deduce a criterion, in terms of the generators of the algebra (up to the integral closure), for $f \in R(V)$ to be bounded on an open semi-algebraic subset of $R(V)$, where $R$ is a real closed field and $V$ is an irreducible real affine variety. This criterion is the following:\\

\emph{ Let $V$ be an irreducible real affine variety and let $h \in R(V)$. Then $h$ is bounded on $S_{\bar p}$ if and only if $h$ is in the integral closure of $B$, where $B$ is the $R$-algebra generated by \[I_{\bar p} = \left\{ \frac{1}{1 + f}\ :\ f \in Cone({\bar p}) \right\}.\] }

We may also obtain such a criterion for fields which are dense in their real closure, such as $\mathbb{Q}$. \\

\emph{ Let R be an ordered field which is dense in its real closure, and let $h \in R(\bar x)$. Then $h$ is bounded on $S_{\bar p}$ if and only if $h$ is in the integral closure of $B$, where $B$ is the $R$-algebra generated by \[I_{\bar p} = \left\{ \frac{1}{1 + f}\ :\ f \in Cone({\bar p}) \right\}.\]} 

\subsection*{Remark of the editor}
All the localizations of $O_K$-algebra of $L=K(V)$ considered in this paper contain the real holomorphy ring $H$ of $L$. In Springer LNM 959, page 21, Theorem 2.16 it was shown by E.Becker that $H $is a Pr\"ufer domain. Therefore every over-ring of $H$ in $L$ is integrally closed (see Larsen-McCarthy: Multiplicative theory of ideals, page 134, Theorem 6.13). This applies in particular to the localizations $A_T$ in Definition $4.1$.

\section {Preliminaries} 

Given a valued field $\langle K,\nu \rangle$ we denote its value group by $\Gamma_\nu$, its valuation ring by $O_\nu = \{a\in K: \nu(a)\ge 0\}$, and its ideal of `infinitesimals' by ${\cal{M}}_\nu = \{a\in K: \nu(a) > 0\}$.  We might also use $\Gamma_K$, $O_K$ and ${\cal{M}}_K$ when the valuation $\nu$ is clear from context. 

We begin by defining the class of valued fields which is our object of interest.

\begin{definition}
An {\em ordered valued field} (i.e, $OVF$) is an ordered field equipped with a valuation $\nu$ satisfying $\forall x,y\in K: 0<x<y \Rightarrow \nu(x)\geq \nu(y)$.
\end{definition}

Equivalently one can require the valuation ring $O_\nu$ to be convex with respect to $\le$.  If $\langle K, \nu \rangle$ is some valued field then we will say that an order $\le_K $ on $K$ is {\em compatible with $\nu$} if $\langle K, \nu, \le_K \rangle \models OVF $. 

We now define a class of $OVF$ which satisfy a nice geometric condition.

\begin{definition}
A {\em real closed valued field} (i.e, $RCVF$) is an ordered valued field which is real closed, such that the valuation is non-trivial.
\end{definition}

Cherlin and Dickmann \cite{CD} proved that the theory $RCVF$ is the model companion of $OVF$.  The relevant consequence for us is that if $K$ is a $RCVF$, and $L$ is an $OVF$ extending $K$, then $K$ is {\em existentially closed} in $L$.  This means that if $\phi(\bar{X})$ is some quantifier free first-order formula with parameters from $K$ in the language $L_{ring}$ enriched by symbols for the valuation and the order, and $L\models \exists \bar{X}:\phi(\bar X)$ then there is some $\bar b\in K^n$ satisfying $\phi(\bar b)$.

\section {$OVF$-integrality} 
 
We say that a rational function $f$ over a valued field $\langle K,\nu \rangle$ is integral at some point $\bar b\in K^n$ if the function $f$ is defined at $\bar b$, and $f(\bar b)$ is in the valuation ring $O_\nu$. A motivation for a refined notion of integrality of a rational function at a point, that is, integrality at a point where $f$ is not defined, was given in section $2.2$ of \cite{DY}. 

Let $V$ be some irreducible real affine variety defined on $K$, and let $L$ denote the field of functions $K(V)$.  

\begin{definition}  \cite{DY}  \label{OVF-val}  
A valuation $\tilde{\nu}$ on $L$ which extends $\nu$ is called an {\em $OVF$-valuation} if there exists some order $\le_L$ on $L$ such that $\langle L,\tilde{\nu}, \le_L\rangle \models OVF $. 
\end{definition} 

An equivalent condition is $\frac{O_{\tilde\nu}}{{\cal M}_{\tilde\nu} }$ being formally real. Note that by definition any $OVF$-valuation on $L$ extends $\nu$. 

\begin{definition}  \cite{DY}  \label{valuation near point}
Let $\langle{K},\nu \rangle \models RCVF$, $\bar b \in V$. We will say that an OVF-valuation $\tilde\nu$ on $L$ is {\em near $\bar b$} if for every $f\in L$ such that $f(\bar b)=0$ and every $\gamma \in \Gamma_K$ we have $\tilde\nu(f)>\gamma$.
\end{definition}

For any $\bar b\in V$ there exist $OVF$-valuations near $\bar b$. Extend for example $\nu$ on $L=K(V)=K(\bar x+I(V) - \bar b)$ by setting $\tilde\nu(x_{i+1}+I(V)-b_{i+1}) > \tilde\nu K(x_1+I(V),...,x_i+I(V))$. The residue field doesn't change, hence $\tilde\nu$ is an $OVF$-valuation. 

\begin{definition}  \cite{DY}  \label{ovf-integrality}  
Let $\langle K, \nu \rangle \models RCVF $. Given $f\in L$ and $\bar b\in V$ we say that $f$ is {\em $OVF$-integral at $\bar b$} if for any OVF-valuation $\nu_{\bar b}$ on $L$ which is near $\bar b$ we have $\nu_{\bar b}(f)\ge 0$. 

For $S \subseteq K^n$ we say that $f$ is {\em $OVF$-integral on $S$} if $f$ is $OVF$-integral at ${\bar b}$ for every ${\bar b} \in S$.
\end{definition}

By existence of $OVF$-valuations near ${\bar b}$ it is easy to conclude that $OVF$-integrality is equivalent to naive integrality whenever $f$ is defined at ${\bar b}$. 
\begin{lemma} \label{integrality eq} \cite{DY}
Let $\langle K, \nu \rangle \models RCVF $ and let $f \in{L}$ and ${\bar b} \in V$ such that $f$ is defined at ${\bar b}$. Then $f$ is integral at ${\bar b}$ if and only if it is $OVF$-integral at ${\bar b}$. 
\end{lemma}
The above follows immediately from the fact that for any $f \in L$ it holds that $\nu_{\bar b}(f)=\nu(f({\bar b}))$ whenever $f$ is defined and not zero at $\bar b$ and $\nu_{\bar b}$ is a valuation near $\bar b$.

\subsection{Properties of $OVF$-integrality}

We begin with a nice lemma which demonstrates the consequences of being an $OVF$-valuation near ${\bar b}$.

\begin{lemma} \label {nearPos}
Let $\langle K, \nu \rangle \models RCVF$, let $\nu_{\bar b}$ be an $OVF$-valuation on $L$ near ${\bar b}\in{V}$, and assume $p\in L$ satisfies $p({\bar b})>0$.  Then for every ordering $\le_L$ on $L$ which is compatible with $\nu_{\bar b}$ we have $p >_L 0$.
\end{lemma}

\begin{proof}
Assuming for a contradiction that $p\le_L 0$ we get $0 < p({\bar b})\le_L p({\bar b})-p$, and by the $OVF$ axiom we get $\gamma:=\nu(p({\bar b}))\ge \nu_{\bar b} \big(p({\bar b})-p\big)$.  However $\big(p({\bar b})-p\big)({\bar b})=0$ , and since $\nu_{\bar b}$ is a valuation near ${\bar b}$ the valuation $\nu_{\bar b}\big(p({\bar b})-p\big)$ is larger than any element of $\Gamma_K$, contradicting $\gamma\in\Gamma_K$. $\qed$
\end{proof}

\begin{remark}
(i)  The reverse implication is false, of course: even if $p\ge_L 0$ for every order $\le_L$ compatible with $\nu_{\bar b}$ we may only deduce $p({\bar b})\ge 0$ (for example consider $p(x)=(x-b)^2$).\\
(ii)  Lemma \ref{nearPos} was implicitly used in the proof of Theorem 4.12 of \cite{DY}, when justifying the necessity property. \\
(iii) Note that the proof is valid also for $\langle K, \nu, \le_K \rangle \models OVF$, but then we shall require that $\le_L$ extends $\le_K$.
\end{remark}

We recall a notation for the positive cone generated by a subset of some field.
\begin{definition} \cite{efr}
Let $L$ be any field, $P\subseteq L$ some subset.  The \emph{positive cone} of $P$ in $L$ is the minimal set $Cone(P)\subseteq L$ containing $P\cup (L)^2 $ which is closed under addition and multiplication. 
\end{definition}
  
Clearly if $\langle L,\le \rangle$ is an ordered field and $P$ is contained in the set $L^{\ge 0}$ of non-negative elements then $Cone(P)\subseteq L^{\ge 0}$.  Note that $L$ is formally real exactly when $-1\notin Cone(\emptyset)$, and that usually the term `cone' is used only when $-1$ is not in  $Cone(P)$.

\begin{proposition}  \label{genOVFint}
Let $\langle K, \nu, \le_K \rangle \models RCVF$. Given polynomials ${\bar p}=(p_1,\ldots,p_m)$ from $K[{V}]$, we define \[S_{\bar p} = \{{\bar b}\in V \ :\ \bigwedge_{i=1}^m p_i({\bar b})>0 \}.\] Then for every $f\in Cone({\bar p})$ the function $\frac{1}{1+f}$ is $OVF$-integral on $S_{\bar p}$.
\end{proposition}
 
\begin{proof}
Fix some ${\bar b}\in S_{\bar p}$, and let $\nu_{\bar b}$ be any $OVF$-valuation near ${\bar b}$ on $L=K(V) $.  We need to show that $\nu_{\bar b}(\frac{1}{1+f})\ge 0$.  Now choose some ordering $\le_L$ on $L$ which is compatible with $\nu_{\bar b}$.  By Lemma \ref{nearPos} we obtain $p_i\ge_L 0$ ($1\le i\le m$), and since $f$ is in the cone generated by the polynomials $p_i$ we also get $f\ge_L 0$.  Therefore we have $0<1\le_L 1+f$, and we may conclude that $\nu_{\bar b}(1)\ge\nu_{\bar b}(1+f)$, or $\frac{1}{1+f} \in O_{\nu_{\bar b}}$, as required. $\qed$
\end{proof}

We now prove that being $OVF$-integral on a definable set coincides with admitting only integral values on it. 

\begin{proposition} \label{strong int}
Let $\langle K, \nu, \le_K\rangle \models RCVF$. Let ${\bar p}=(p_1,\ldots,p_m), {\bar g}=(g_1,\ldots,g_l)$ where $p_1,\ldots,p_m \in K[{V}] $ and $g_1,\ldots,g_l \in K({V}) $. Let \[S_{{\bar p},{\bar g}} = \left\{{\bar b} \in V \ : \ \bigwedge_{i=1}^m p_i({\bar b})>0, \bigwedge_{j=1}^l \nu(g_j({\bar b}))\ge 0\right\}. \]  Then for every $h \in K(V)$ we have that $h$ is $OVF$-integral on $S_{\bar p, \bar g}$ if and only if $h$ admits only integral values on $S_{\bar p, \bar g}$.
\end{proposition}
\begin{proof}
One direction follows directly from Lemma \ref{integrality eq}. For the other direction, let ${\bar b} \in S_{\bar p, \bar g}$ such that $h$ is not defined at ${\bar b}$ and suppose that $h$ is not $OVF$-integral at ${\bar b}$. Hence there exists $\nu_{\bar b}$ - an $OVF$-valuation which extends $\nu$ to $K(V)$ near ${\bar b}$ - such that $\nu_{\bar b}(h)<0$. Let $\le_{\bar b}$ be an order on $K(V)$ compatible with $\nu_{\bar b}$. According to Lemma \ref{nearPos} every such order satisfies ${\bar p} > 0 $. Hence, $\langle {{K}}({V}), \nu_b, \le_b \rangle \models \bigwedge_{i=1}^m p_i(\bar X)>0 \wedge \bigwedge_{i=1}^l \nu(q_i(\bar X)) \ge 0  \wedge \bigwedge_{i=1}^s q_i(\bar X)=0$, where $q_1, \ldotp q_s$ generate $I(V)$. Since ${K}$ is existentially closed in ${{K}}({V})$ there exists some $d \in S_{\bar p, \bar g}$ such that $\nu(h(d))<0 $. Hence $h$ doesn't admit only integral values on $S_{\bar p, \bar g}$. $\qed$
\end{proof} 

\section{Formulation of the general Ganzstellensatz}

We introduce here the definitions and the framework for proving a Ganzstellensatz. A motivation for these definitions, as well as a proof of Proposition \ref{fw} for a general theory of valued fields, can be found in \cite{DY}.

Let $\langle K,\nu \rangle$ be some valued field, $L$ a field extension of $K$. Let $A\subseteq L$ be some $O_K$-algebra such that $A \cap K = O_K $. Define $T=\{1+ma\ :\ m\in {\cal{M}}_K, a\in A\}$, and note that $T$ is a multiplicative set.

\begin{definition} \label{intT}
The \emph{integral radical} of $A$ in $L$ is defined as the integral closure in $L$ of the localization $A_T$, and will be denoted by $\sqrt[int]{A}$.
\end{definition}

Note that $\langle K,\nu \rangle$ and $L$ are omitted from the above notation for convenience.  

The following proposition is a special case of the model theoretic framework given in Lemma $2.17$ of \cite{DY}, with the simple modification of $K(V)$ instead of $K(\bar x)$. 
 
\begin{proposition}  \cite{DY}  \label{fw}
Let $\langle {K},\nu \rangle \models RCVF $ and let $S \subseteq {K}^n$ be a nonempty set defined by a formula $\phi_S({\bar X})$ in the language of ring enriched by symbols for the valuation and order. Let $L=K(V)$ where $V$ is an irreducible real affine variety over $K$, and let $A\subseteq L$ be some $O_K$-algebra such that $A \cap K = O_K $, which also admits the following properties: 
\\
\textbf{Necessity} If $S \subseteq V$ then every $f\in A$ is $OVF$-integral on $S$. 
\\
\textbf{Sufficiency}  Let $\tilde\nu $ be any valuation on $L$ extending $\nu$ such that $A \subseteq O_{\tilde\nu} $. Then there exists an order $\le_L$ on $L$ which is compatible with $\tilde\nu$, and such that $\langle {L},\tilde\nu, \le_L \rangle \models \phi_S({\bar x +I(V)}) $. \\

If $S \subseteq V$, then for every $h \in L $ we have that $h$ is $OVF$-integral on $S$ if and only if $h \in \sqrt[int]{A}$.
\end{proposition} 

As the proposition has been slightly modified, we would like to reprove it here. In order to do so, we need the following Lemma from \cite{ko}, which we shall also use further in the paper.

\begin{lemma} \label{kochenl}
Let ${\langle K, \nu \rangle}$ be a valued field, ${L}$ a field extension of ${K}$ and $A$ a sub-ring of ${L}$ such that $A\cap {K} = O_K $. Then $\sqrt[int]{A} $ is the intersection of all valuation sub-rings $O_L $ of ${L}$ such that $A \subseteq O_L  $ and $O_L \cap {K} = O_K $.
\end{lemma}

\begin{proof} [proof of \ref{fw}.]
Suppose that $S \subseteq V$, as otherwise there is nothing to prove. One inclusion is obvious and follows directly from the necessity condition and some straightforward calculation. For the other direction, let $h \notin \sqrt[int]{A}$. By Lemma \ref{kochenl} there exists some $\tilde\nu$ a prolongation of $\nu$ to $L$ such that $A \subseteq O_{\tilde\nu}$ and $\tilde\nu(h)<0$. Then, by the sufficiency condition, there exists some $\le_L$ compatible with $\tilde\nu$, such that $\langle L, \tilde\nu, \le_L \rangle \models \phi_S(\bar x + I(V)) \wedge \nu(h)<0$. Since $K$ is existentially closed in $L$, there exists some $\bar b \in K^n$ such that $\langle K, \nu, \le_K \rangle \models \phi_S(\bar b) \wedge \nu(h(\bar b))<0$. Hence, $\bar b \in S$ and $\nu(h(\bar b))<0$ hence $h$ is not $OVF$-integral on $S$. $\qed$ 
\end{proof}

\section{A Ganzstellensatz for open semi-algebraic sets}

Let $K$ be a real closed valued field, $V$ some irreducible real affine variety over $K$, and $L=K(V)$. Let ${\bar p}=(p_1,\ldots,p_m)$ be a tuple of polynomials from $K[{V}]$, and let \[S_{\bar p} = \{{\bar b}\in V\ :\ \bigwedge_{i=1}^m p_i({\bar b})>0\}.\]  Assume that the set $S_{\bar p}$ is non-empty. 

Recall that $Cone({\bar p})$ in $L$, where $K[V] \subseteq L$, denotes the positive cone generated by the polynomials $p_i$. Since we assumed $S_{\bar p}\neq \emptyset$ we get $-1\notin Cone({\bar p})$, and we let $I_{\bar p} = \{\frac{1}{1 + f}\ :\  f\in Cone({\bar p})\}$. Considering the representation of a general element in $Cone(\bar p)$,  note that if we denote by $Q$ the set of sums of squares in $L$ then
$$I_{\bar p} = \Big\{ \big( 1 + \sum_{J\subseteq[m]} r_J \prod_{i\in J}p_i \big)^{-1} \ : \  r_J\in Q \Big\}.$$

Finally let $A_{\bar p}$ be the $O_K$-algebra generated by $I_{\bar p}$. Our goal in this section is to show that the $O_K$-algebra $A_{\bar p}$ satisfies the conditions of Proposition \ref{fw}.  We start with a few lemmas.  
   
\begin{lemma} \label {evenCase}
Let $\tilde\nu $ be any valuation on $L$ extending $\nu$ such that $A_{\bar p} \subseteq O_{\tilde\nu}$, and let $C = \{res(\frac{q}{c^2}) : q \in \langle p_1,...,p_m \rangle, c\in L, \tilde\nu(q)=\tilde\nu(c^2) \} $ where $\langle p_1,...,p_m \rangle $ is the multiplicative semi-group generated by $p_1,...,p_m $. Let $\ell$ denote the residue field of $(L,\tilde\nu)$. 

Then there exists an order $\le_\ell$ on $\ell$ such that $f \ge_\ell 0$ for every $f \in C $.
\end{lemma}
\begin{proof}
We shall assume by way of contradiction that there is no such linear order, i.e, $-1 $ belongs to the cone generated by $\{ a^2 \ : \ a \in \ell\} \cup C $. Hence, there exist some $q_1,...,q_t \in \langle p_1,...,p_m \rangle $ and some $c_1,...,c_t \in {L} $ such that $\tilde\nu(c_i^2)=\tilde\nu(q_i) $, and there exist some $\rho_1,...,\rho_t \in \ell$ such that every $\rho_j $ is a sum of squares in $\ell$, satisfying \[
 \sum_{j=1}^{t} \rho_j res({\frac{q_j}{c_j^2}} )=-1 
.\] 
Let $r_1,...,r_t \in O_{\tilde{\nu}} $ be such that every $r_j $ is a sum of square elements in ${L} $ and $res(r_j)=\rho_j $. Then \[
1+\sum_{j=1}^t r_j  {(\frac{q_j}{c_j^2})} \in {\cal{M}}_L .
\]

Therefore the inverse of the above expression has negative valuation.  However this inverse clearly has the form $\frac{1}{1 + f}$ for some $f\in Cone({\bar p})$, contradicting $I_{\bar p}\subseteq O_{\tilde\nu}$.  $\qed$
\end{proof}

We now present the main tool for constructing an order on a valued field from the order on its residue field. We begin with the definition of semi-sections, which exist for any valued field, as Lemma \ref{semisec} will show.

\begin{definition}    
A {\em semi-section} of a valued field $\langle L,\mu \rangle $ is a map $s:\Gamma_L \rightarrow L^{\times}$ such that for any $\gamma \in \Gamma_L $ we have $\mu(s(\gamma))=\gamma $ and for every $ \gamma_1,\gamma_2 \in \Gamma_L $ we have \[\frac{s(\gamma_1+\gamma_2)}{s(\gamma_1)s(\gamma_2)} \in {L^{\times}}^2 .\]
\end{definition}

The following lemma is a special case of the Baer-Krull theorem, see for example \cite{efr}. We use it in this paper to construct from a semi-section and an order on the residue field, an order on the field which induces the order on the residue field.

\begin{lemma} \label{bk}
Let $\langle L, \mu \rangle$ be a valued field, and let $\le_\ell$ be an order on the residue field $\ell$.  For any semi-section $s$ of $\langle L, \mu \rangle$ we can define an order on $L$ by $x >_L 0 \Leftrightarrow res\big(\frac{x}{s(\mu(x))}\big)>_\ell 0$. Moreover, the order $\le_L$ induces $\le_\ell$. Finally, any order on $L$ which induces $\le_\ell$ is compatible with $\mu$.   
\end{lemma}

\begin{lemma} \label {semisec}
Let $\tilde\nu $ be a valuation on ${L} $ extending $\nu $, and let $\le_\ell$ be an order on the residue field $\ell$ of $\langle L,\tilde\nu \rangle$.  Given $p_1,...,p_t \in{L}$ let $\gamma_i = \tilde\nu(p_i) \in \Gamma_L$ be their corresponding valuations.  Assume that $\big\{\frac {\gamma_i}{2\Gamma_L}\big\}_{i=1}^t $ are linearly independent over $\frac {\mathbb{Z}}{2\mathbb{Z}}$.  Then there exists an order $\le_L$ which induces $\le_\ell$, such that $p_i \ge_L 0 $ for every $1\le i \le t $.
\end{lemma}

\begin{proof}
Let $\tilde\gamma_i=\frac {\gamma_i}{2\Gamma_L}  $ for $1 \le i \le t $ and
let $\{\tilde\gamma_i\}_{i\in \delta}$ be an extension to a base of $\frac {\Gamma_L} {2\Gamma_L}$ as a vector space over $\frac {\mathbb{Z}}{2\mathbb{Z}}$. Let $\tilde s : \frac {\Gamma_L} {2\Gamma_L} \rightarrow \frac {{L}^{\times}} {{{L}^{\times}}^2} $ be a group homomorphism such that $\tilde{s}(\tilde\gamma_i)=\tilde{p_i}$ for $1\le i\le t$ and $\tilde {p_i}$ is the image of $p_i$ in $\frac {{L}^{\times}} {{{L}^{\times}}^2} $. Our aim is to find a function $s$ such that the following diagram commutes and $\tilde\nu(g)=\tilde\nu(s(\tilde\nu(g))) $ for every $g \in {L} $. \\
\begin{center}
    $\begin{CD}
\Gamma_L @>s>> {L}^{\times} \\
@VV\pi_{2\Gamma_L} V     @VV\pi_ {{{L}^{\times}}^2} V\\
\frac {\Gamma_L} {2\Gamma_L} @>\tilde s>> \frac {{L}^{\times}} {{{L}^{\times}}^2}
\end{CD}$ \\
\end{center}
For every $\gamma \in 2\Gamma_L $ let $f_{\gamma} \in {L^{\times}}^2$ be such that $\tilde\nu(f_{\gamma})=\gamma $. Let $\{\gamma_i\}_{i \in \delta} $ be representatives for $\{\tilde\gamma_i\}_{i\in \delta} $ and let $f_{\gamma_i} \in L^{\times} $ be such that $\tilde\nu(f_{\gamma_i})=\gamma_i $ and lifting $\tilde{s}(\tilde\gamma_i) $. Define $s(\sum_{j=1}^t \gamma_{i_j} + 2\gamma )=\prod_{j=1}^t f_{\gamma_{i_j}}f_{2\gamma} $. Since every element in $\Gamma_L$ has a unique such representation $s$ is well defined. 
Thus $s$ is a semi-section, and by Lemma \ref{bk} we may define an order $\le_{L}$ on $L$ by  $g >_L 0 \Leftrightarrow res(\frac{g}{s(\tilde\nu(g))})>_\ell 0 $, and $\le_L$ induces $\le_\ell$. For every $1\le i \le t $ we have $s(\gamma_i)=p_ic_i^2 $ for some $c_i \in L$ with $\tilde\nu(c_i)=0 $, hence $res(\frac{p_i}{s(\gamma_i)})>0 $ and $p_i>_L 0 $. $\qed$
\end{proof}

Now, we are ready to prove the main theorem of the paper.

\begin{theorem} \label{ganz}
Let $\langle K, \nu \rangle $ be a real closed valued field, $V$ some irreducible real variety, and $L=K({V}) $. Let ${\bar p}=(p_1,..,p_m)$ where $p_1,...,p_m \in K[V]$.  Let \[S_{\bar p} = \{{\bar b}\in V\ :\ \bigwedge_{i=1}^m p_i({\bar b})>0\}. \] 

Let $I_{\bar p} = \left\{ \frac{1}{1 + f}\ :\ f \in Cone({\bar p}) \right\}$, and let $A_{\bar p}$ be the $O_K$-algebra generated by $I_{\bar p}$. \\

Then for every $h\in L$, $h\in \sqrt[int] {A_{\bar p}}$ if and only if $h$ admits only integral values on $S_{\bar p}$. 
\end{theorem}

\begin{proof}
In order to prove the theorem we shall prove that $A_{\bar p}$ satisfies the conditions of Proposition \ref{fw}. 

For the necessity condition we need to prove that for every $h \in A_{\bar p}$, $h$ is $OVF$-integral on $S_{\bar p}$. Since $A_{\bar p}$ is generated as an $O_K $-algebra by $I_{\bar p}$ it will be enough to prove that $h$ is $OVF$-integral on $S_{\bar p}$ for every $h \in I_{\bar p}$.  This is exactly the content of Proposition~\ref{genOVFint}.

If $V \subseteq K^n$ and $I(V)$ is generated by the polynomials $q_1, \ldots , q_l$, Then $S_{\bar p}$ may be identified to \[\Big\{ \bar b : \bigwedge_{i=1}^l q_i(\bar b)=0, \bigwedge_{i=1}^m p_i(\bar b) > 0 \Big\} \]
where we allow ourselves to call still $p_i$ any representative in $K[\bar x]$ of $p_i$. \\
In order to prove the sufficiency condition, we need to show that for every valuation  ${\tilde{\nu}}$ extending $\nu$ to ${L}$ such that $A_{\bar p}\subset O_{\tilde{\nu}}$, there exists an order $\le_L$ on $L$, compatible with $\tilde\nu$, such that $p_1,...,p_m > 0$. Where, of course, $\phi_S$ is the formula ``$\bigwedge_{i=1}^m p_i(\bar X)>0 \wedge \bigwedge_{i=1}^l q_i(\bar X)=0$''.
  
For $1\le i\le m$ let $\gamma_i = \tilde\nu(p_i)$, and denote by $\tilde{\gamma_i}$ the image of $\gamma_i$ in $\frac {\Gamma_L} {2\Gamma_L} $. Without loss of generality, let $p_1,...,p_t$ be such that $\{\tilde{\gamma_1},...,\tilde{\gamma_t}\}$ is a maximal independent subset of $\{\tilde{\gamma_1},...,\tilde{\gamma_m}\}$, over $\frac {\mathbb{Z}}{2\mathbb{Z}}$. So, for all $t+1 \le i \le m $  \[\tilde\nu(p_i\prod_{j=1}^t p_j^{s_{i,j}}) \in 2\Gamma_L\] for some $s_{i,j} \in \{0,1\} $. Therefore, by Lemma \ref{evenCase} and Lemma \ref{nearPos}, there exists an order $\le_\ell$ on the residue field $\ell$ of $L$ such that, for any order $\le_L $ on $L$ which induces $\le_\ell$, all elements in the set 
\[ \left\{\frac {p_i\prod_{j=1}^t p_j^{s_j}}{c_i^2}\ \ : \ t+1\le i\le m,c_i \in L, \tilde\nu(c_i^2)=\tilde\nu(p_i\prod_{j=1}^t p_j^{s_j})\right\}\] are strictly positive. \\
By Lemma \ref{semisec} there exists an order $\le_L$ inducing $\le_\ell$ such that $p_1,...,p_t\ge_L 0$. Hence $ {p_i\prod_{j=1}^t p_j^{s_{i,j}}}\ge_L 0$ for every $t+1\le i\le m $. Thus $p_1,...,p_m \ge_L 0$.  By the last part of Lemma~\ref{bk} any order on $L$ which induces an order on the residue field, in particular $\le_L$ is compatible with $\tilde\nu$. Hence, the sufficiency property holds. \\
As proved above, $A_{\bar p}$ satisfies the necessity and sufficiency properties, and therefore by Proposition \ref{fw}, $h\in \sqrt[int] {A_{\bar p}}$ if and only if $h$ is $OVF$-integral on $S_{\bar p}$. Hence, according to Proposition \ref{strong int},  $h\in \sqrt[int] {A_{\bar p}}$ if and only if $h$ admits values only in the valuation ring on $S_{\bar p}$.  $\qed$
\end{proof}

We also get a similar Ganzstellensatz for the intersection of $S_{\bar p}$ with finitely many valuation inequalities, with an almost identical proof.

\begin{theorem} \label{ganz2}
Let $\langle K, \nu \rangle $ be a real closed valued field and let $L=K({V}) $ where $V$ is some irreducible real affine variety. Let ${\bar p}=(p_1,\ldots,p_m), {\bar g}=(g_1,\ldots,g_l)$ where $p_1,\ldots,p_m \in K[{V}] $ and $g_1,\ldots,g_l \in K({V}) $. Let \[S_{{\bar p},{\bar g}} = \left\{{\bar b} \in V \ : \ \bigwedge_{i=1}^m p_i({\bar b})>0, \bigwedge_{j=1}^l \nu(g_j({\bar b}))\ge 0\right\}. \] 

Let $I_{\bar p} = \left\{ \frac{1}{1 + f}\ :\ f \in Cone({\bar p}) \right\}$, and
let $A_{{\bar p},{\bar g}}$ be the $O_K$-algebra generated by $ I_{\bar p} \cup \{g_1,...,g_l\}$ in $L$. \\
Let $h\in L$. Then $h\in \sqrt[int] {A_{{\bar p},{\bar g}}}$ if and only if $h$ is $OVF$-integral over $S_{{\bar p},{\bar g}}$. 
\end{theorem}

In \cite{DYG}, a characterization of the rational functions over an algebraically closed field, that are defined on some algebraic set intersected by some valuative semi-algebraic set, and admit values only in the valuation ring, is given. Here, we show that their proof may apply also for real closed fields. Combining it with the main theorem of this paper, Theorem \ref{ganz}, we may deduce the boundedness characterization that we wish to obtain. \\
\begin{theorem} \label{zerogen}
Let $K$ be a real closed valued field, and let \[S_{{\bar p},{\bar q},{\bar f}} = \left\{{\bar b} \in {K}^n \ : \ \bigwedge_{i=1}^m p_i({\bar b})>0, \bigwedge_{j=1}^l q_j({\bar b})= 0, \bigwedge_{t=1}^s \nu(f_t(b)) \ge 0\right\}. \] Suppose that $S_{{\bar p},{\bar q},{\bar f}} \neq \emptyset $. Then for every $h \in K(\bar x)$ such that $h$ is defined on $S_{{\bar p},{\bar q},{\bar f}}$, we have that $h(S_{{\bar p},{\bar q},{\bar f}}) \subseteq O_K $ if and only if $h$ is in the integral closure of ${A}$, where $A$ is the $O_K$-algebra generated by $I(V({\bar q})) \cup I_{\bar p} \cup {\bar f}$, and $I(V({\bar q}))$ is the real ideal generated by $\bar q$.
\end{theorem}
\begin{remark}
Note that when $I(V({\bar q}))$ is prime, Theorem \ref{ganz2} already supplies the required characterization, without limiting the functions to be defined at every point in the set. However, when it is not prime, the requirement for the function to be defined at every point is indeed not redundant. For example, let $S$ be the zero set of $xy$ in $R^2$ for some real closed valued field $R$. The function $\frac{x}{y}$ admits only values in the valuation ring where it is defined. However, for any valuation $\tilde\nu$ near $(1,0)$, we have that the ideal generated by $y$ is in $O_{\tilde\nu}$, and $\tilde\nu(\frac{x}{y})<0$. Hence, $\frac{x}{y} \notin \sqrt[int]{I(S) \cup \{ \frac{1}{1+f^2} \: \ f \in K(\bar x) \} }$.
\end{remark}
\
\begin{proof} 
One direction is trivial, since any element of $A$ admits only values in the valuation ring. For the other direction, let $h \notin\sqrt[int]{A}$. Hence, by Lemma \ref{kochenl}, there exists a prolongation $\tilde\nu$ on $K(\bar x)$ such that $A \subseteq O_{\tilde\nu}$ and $\tilde\nu(h)<0$. Let \[I_{\tilde\nu}= \left\{ g \in K[\bar x] \ :\ \exists q \in I(V({\bar q})), \exists n \in \mathbb{N}, \tilde\nu(g) \ge n\tilde\nu(q) \right\} ,\] exactly as defined in the first theorem of \cite{DYG}. For every $q \in I(V({\bar q}))$ and for every $c \in K $, we have that $\tilde\nu(cq) \ge 0$, since $I(V({\bar q})) \subseteq A$. Hence, for every $q \in I(V({\bar q}))$, we have that $\tilde\nu(q) > \Gamma_K$. Hence, $1 \notin I_{\tilde\nu}$. Therefore, $I_{\tilde\nu}$ is clearly a proper ideal. It is prime, since if $q,q' \notin I(V({\bar q}))$ then for every $s \in I(V({\bar q}))$ and every $n \in \mathbb{N}$, we have that $\tilde\nu(qq') < 2\frac{s}{2n}=\frac{s}{n}$. It is real, since for every $a,b \in K(\bar x)$ we have that $\tilde\nu(a^2), \tilde\nu(b^2) \ge \tilde\nu(a^2+b^2)$. \\ Set $h=\frac{h_1}{h_2}$ where $h_1,h_2 \in K[\bar x]$, $(h_1;h_2)=1 $, $L=K(V(I_{\tilde\nu}))$ and let $\bar{\nu}$ be the valuation that  $\tilde\nu$ induces on $L$. We now would like to use the sufficiency condition in Proposition \ref{fw}, for $\phi_{S_{\bar p, \bar q \bar f}}$. That is, to show that there exists $\le_L$ such that $\langle L, \bar\nu, \le_L \rangle \models OVF$ and that $\langle L, \bar\nu, \le_L \rangle \models \phi_{S_{\bar p, \bar q \bar f}}\wedge \bar\nu(h_1)<\bar\nu(h_2) $.
 In order to make it meaningful and to suffice for proving our theorem, we must have $\bar\nu(h_2)<\infty$, which holds if and only if $h_2 \notin I_{\tilde\nu}$. First, we show that $S_{\bar p,\bar f}\cap V(I_{\tilde\nu})\neq\emptyset$. As in \cite{DYG}, let $\bar t \subset K[\bar x]$ be a set of generators of $I_{\tilde\nu}$. Due to the exact same argument for proving sufficiency in Theorem \ref{ganz} and \ref{ganz2}, there exists some $\le_L$ compatible with $\bar\nu$ such that $\langle L, \bar\nu, \le_L \rangle \models \phi_{S_{\bar p,\bar f}}\wedge \bar t = 0$. This implies that $V(I_{\tilde\nu}) \cap S_{{\bar p},{\bar f}} \neq \emptyset$. Hence, as $h$ is defined on $S_{{\bar p},{\bar q},{\bar f}} $, and in particular on each non-empty subset of it, we have that $h_2 \notin I_{\tilde\nu}$. Hence, we may obtain the sufficiency condition for our definable set. Hence, there exists $\bar b \in S_{{\bar p},{\bar q},{\bar f}} $ such that $\nu(h(b))<0$. $\qed$
\end{proof}

\section {Boundedness of rational functions on open semi-algebraic sets}
We now show how to use the main theorem of the paper, Theorem \ref{ganz}, in order to give a criterion for a rational function to be bounded on an open semi-algebraic set over some real closed field. For every non-archimedean real closed field $R$ there exists a canonical valuation which makes $R$ a model of $RCVF$ whose valuation ring is the convex hull of $\mathbb{Z}$ in $R$, denoted by $O_R$. This valuation ring is also referred to in the literature as the ring of holomorphy of $R$. \\ 
 In order to give a criterion for general boundedness, we first introduce the notion of weak boundedness.

\begin{definition}
Let $K \models RCVF$, let $V$ be an irreducible real affine variety and let $a \in K$. We say that $h \in K(V)$ is \emph {weakly bounded by $a$} on a set $S$ if $h(S) \subseteq O_{a}$, where $O_a=\cup_{n \in \mathbb{N}} a^nO_K$.
\end{definition}

As $O_a$ is a convex valuation ring, we get from Theorem \ref{ganz2} the following corollary.

\begin{corollary}\label{ballbounded}
Let $K \models RCVF$, let $V$ be an irreducible real affine variety and let $a \in K, a \ge 0$. For every $h \in K({V})$, we have that $h$ is weakly bounded by $a$ on $S_{{\bar p},{\bar g}}$ if and only if $h \in \sqrt[int]{{A_a}}$ where $A_a$ denotes the $O_a$-algebra generated by $I_{\bar p} \cup \{g_1,...,g_l\}$.
\end{corollary}
We may now obtain the following result about boundedness of rational functions on open semi-algebraic sets in any real closed valued field, whose value group has no maximal archimedean class.

\begin{corollary} \label{bounded}
Let $K \models RCVF$ such that $\Gamma_K$ has no maximal archimedean class, and let $V$ be an irreducible real affine variety. For every $h \in K({V})$, we have that $h$ is bounded on $S_{{\bar p},{\bar g}}$ if and only if $h$ is in the integral closure of ${{\cal A}}$, where $\cal A$ denotes the $K$-algebra generated by $I_{{\bar p}}\cup \{g_1,..g_l\}$.
\end{corollary}
\begin{proof}
Suppose that $h$ is bounded on $S_{{\bar p},{\bar g}}$ then there exists some $a \in K, a>0$ such that $-a < h(S_{{\bar p},{\bar g}}) <a$. Then obviously $h(S_{{\bar p},{\bar g}}) \subseteq O_{a}$, hence $h$ is weakly bounded by $a$ on $S_{{\bar p},{\bar g}}$. On the other hand, if there exist some $a > 0$ such that $h$ is weakly bounded by $a$ on $S_{{\bar p},{\bar g}}$, then, since $\Gamma$ has no maximal archimedean class, there exists some $c \in K$ such that $c \notin O_{a}$ positive, hence $-c < h(S_{{\bar p},{\bar g}}) < c$, hence $h$ is bounded on $S_{{\bar p},{\bar g}}$. Hence $h$ is bounded on $S_{{\bar p},{\bar g}}$ if and only if it belongs to the increasing union of $\sqrt[int]{A_{a}}$, which is exactly the integral closure of the $K$-algebra generated by $I_{{\bar p},{\bar g}}$. $\qed$
\end{proof}

We now use the compactness theorem in order to reduce the general case of a real closed field to a real closed valued field with a value group which has no maximal archimedean class. That is, we prove that being weakly bounded is equivalent to being bounded, also without the condition that the value group has no last archimedean component. 

\begin{theorem} \label {archi}
Let $R$ be any real closed field, let $V$ be an irreducible real affine variety, let $\bar p \subset K[V]$ and let $h \in R(V)$. Then $h$ is bounded on $S_{\bar p}$ if and only if $h$ is in the integral closure of $B$, where $B$ is the $R$-algebra generated by \[I_{\bar p} = \left\{ \frac{1}{1 + f}\ :\ f \in Cone({\bar p}) \right\}.\] 
\end{theorem}
\begin{proof}
Let $\hat{R}$ be the minimal elementary extension of $R$ realizing the type saying $\phi(X,a)=``0<X<a''$ for every $a \in R$. Let $\tilde{R}$ be the field obtained after $\omega$ such steps. Then $\tilde{R}$ equipped with its canonical valuation is a $RCVF$ model with a value group which has no maximal archimedean class. Suppose that $f \in R(V)$ is bounded on $S_{\bar p}$.  One can express that $f$ gets on $S_{\bar p}$ values bigger than $s$ or smaller than $-s$ for some $s>0$ on $S_{\bar p}$ by a first order formula, hence $h$ is bounded also on $S_{\bar p}(\tilde{R}) $. Hence, according to Corollary \ref{bounded}, $h$ is in the integral closure of $\tilde{B}$, where $\tilde{B}$ is the $\tilde{R}$-algebra generated by $I_{\bar p}$. Now, suppose by contradiction that $h$ is not in the integral closure of $B$. Hence, by Lemma \ref{kochenl}, there exists a valuation $\nu$ on $R(V)$ such that $B \subseteq O_{\nu}$ and $\nu(h)<0$. For every $a \in R^{\times}$ we have that $\nu(a)=0$, as there exist $f \in Cone(\bar p)$ with $\nu(\frac{1}{1+f})=0$. In fact, for every $f$ which is a sum of squares it is the case for either $f$ or $\frac{1}{f}$. Let $\tilde\nu$ be a prolongation of $\nu$ to $\tilde{R}(V)$ such that $\tilde\nu|_{\tilde{R}}$ is trivial. Then $\tilde\nu(h)<0$ and $\tilde{B} \subseteq O_{\tilde\nu}$. Hence, $h$ is not in the integral closure of $\tilde{B}$, which leads to a contradiction. $\qed$
\end{proof}

We may further deduce the above characterization for ordered fields with the property of being dense in their real closure, such as $\mathbb{Q}$.
\begin{corollary}
Let $F$ be an ordered field such that $F$ is dense in its real closure, and let $h \in F(\bar x)$. Then $h$ is bounded on $S_{\bar p}$ if and only if $h$ is in the integral closure of $B$, where $B$ is the $F$-algebra generated by \[I_{\bar p} = \left\{ \frac{1}{1 + f}\ :\ f \in Cone({\bar p}) \right\}.\]  
\end{corollary}
\begin{proof}
One direction is obvious. We are left to prove that if $h$ is bounded on $S_{\bar p}$ then $h$ is in the integral closure of $B$. We repeat the argument of the proof of Theorem \ref{archi}. Let $\tilde{F}$ be the real closure of $F$ and $B(\tilde{F})$ the $\tilde{F}$-algebra generated by \[I_{\bar p} = \left\{ \frac{1}{1 + f}\ :\ f \in Cone({\bar p}) \right\}.\] Suppose by contradiction that there exists some valuation $\nu$ on $F(\bar x)$ such that $B \subseteq O_{\nu}$ and $\nu(h)<0$. As in the proof of Theorem \ref{archi}, let $\tilde\nu$ be a prolongation of $\nu$ to $\tilde{F}(\bar x)$ such that $\nu|_{\tilde{F}}$ is trivial. Hence $\tilde\nu(h)<0$ and $B(\tilde{F})\subseteq O_{\tilde\nu}$, which leads to a contradiction with Theorem \ref{archi} which implies that $h$ is in the integral closure of $B(\tilde{F})$. Hence, $h$ is not bounded on $S_{\bar p}(\tilde{F})$. Since $F$ is dense in $\tilde{F}$ we get that $h$ is not bounded on $S_{\bar p}$. $\qed$
\end{proof}

\section{Acknowledgments}

The main part of this work was submitted as part of the author's M.Sc. thesis, supervised by Assaf Hasson at Ben-Gurion University, Israel.  The author wishes to thank from the bottom of her heart to Yoav Yaffe, for introducing her the question of work and the Baer-Krull ingredient which has been central in this work, and for his devoted guidance which without it the work probably would not have taken place. The main body of this paper (Sections 2-5) represents joint work with (and
was jointly written by) Yoav Yaffe, who supervised informally on the master thesis, and out of his own reasons and time constraints chose not to be a co-author of this paper. The author would like also to thank Jeff Burdges, Antongiulio Fornasiero and Liran Shaul for reading carefully the paper, and giving their useful comments.


\end{document}